\documentclass[12pt,leqno,a4paper,reqno,final]{amsart}
\usepackage{amsmath, amsthm, verbatim, amsfonts, amssymb, xcolor,bbm}
\usepackage{mathrsfs}
\usepackage{dsfont}
\usepackage[a4paper,margin=1in]{geometry}

\usepackage[english]{babel}
\usepackage[notref,notcite]{showkeys}
\usepackage[german=guillemets]{csquotes}
\usepackage{url}
\setquotestyle[british]{english}

\theoremstyle{plain}
\newtheorem{theorem}{Theorem}
\newtheorem{corollary}[theorem]{Corollary}
\newtheorem{lemma}[theorem]{Lemma}

\newtheorem*{fact}{Fact}

\newtheorem*{theorem*}{Theorem}
\newtheorem*{corollary*}{Corollary}
\newtheorem*{lemma*}{Lemma}
\newtheorem*{proposition*}{Proposition}

\theoremstyle{definition}
\newtheorem{remark}[theorem]{Remark}
\newtheorem*{remark*}{Remark}
\newtheorem{example}[theorem]{Example}
\newtheorem{definition}[theorem]{Definition}

\newtheorem*{notation}{Notation}

\renewcommand\labelenumi{\textup{\alph{enumi})}}
\renewcommand\theenumi\labelenumi

\makeatletter
\def\@makefnmark{\hbox{\textup{(}\@textsuperscript{\normalfont\@thefnmark}\textup{)}}}
\makeatother

\DeclareFontFamily{U}{mathx}{\hyphenchar\font45}
\DeclareFontShape{U}{mathx}{m}{n}{
      <5> <6> <7> <8> <9> <10>
      <10.95> <12> <14.4> <17.28> <20.74> <24.88>
      mathx10
      }{}
\DeclareSymbolFont{mathx}{U}{mathx}{m}{n}
\DeclareFontSubstitution{U}{mathx}{m}{n}
\DeclareMathAccent{\widecheck}{0}{mathx}{"71}
\DeclareMathAccent{\wideparen}{0}{mathx}{"75}

\newcommand\Ee{\mathds{E}}

\newcommand\real{{\mathds{R}}}

\newcommand\I{\mathds{1}}
\newcommand\Pp{\mathds{P}}
\newcommand\id{\mathop{\mathrm{id}}}

\newcommand\supp{\mathop{\mathrm{supp}}}

\newcommand\esssup{\mathop{\mathrm{ess\, sup}}}
\newcommand\re{\mathop{\mathrm{Re}}}

\newcommand\rn{{\mathds{R}^n}}

\newcommand{\Lcal}{\mathcal{L}}
\newcommand{\Fcal}{\mathcal{F}}

\newcommand{\Pcal}{\mathcal{P}}
\newcommand{\Scal}{\mathcal{S}}

\newcommand{\boxperp}{[\perp]}

\renewcommand\phi\varphi

\newcommand{\entier}[1]{\lfloor #1\rfloor}
\newcommand{\scalp}[2]{\langle #1,\,#2\rangle}

\usepackage{marginnote}
\begin{document}
\title[Liouville, Fourier multipliers and coupling]{\bfseries The Liouville theorem for a class of Fourier multipliers and its connection to coupling}

\author[D.~Berger]{David Berger}
\author[R.L.~Schilling]{Ren\'e L.\ Schilling}
\address[D.~Berger \& R.L.\ Schilling]{TU Dresden\\ Fakult\"{a}t Mathematik\\ Institut f\"{u}r Mathematische Stochastik\\ 01062 Dresden, Germany}
\email{david.berger2@tu-dresden.de}
\email{rene.schilling@tu-dresden.de}

\author[E.~Shargorodsky]{Eugene Shargorodsky}
\address[E.~Shargorodsky]{King's College London\\ Strand Campus\\ Department of Mathematics\\ Strand, London, WC2R 2LS, U.K.  \emph{and} TU Dresden\\ Fakult\"{a}t Mathematik\\ Institut f\"{u}r Mathematische Stochastik\\ 01062 Dresden, Germany}
\email{eugene.shargorodsky@kcl.ac.uk}

\thanks{\emph{Acknowledgement}. We are grateful for the careful comments of several expert referees, who helped us to improve the exposition of this paper. Financial support for the first two authors through the DFG-NCN Beethoven Classic 3 project SCHI419/11-1 \& NCN 2018/31/G/ST1/02252, the  6G-life project (BMBF programme ``Souver\"an. Digital. Vernetzt.'' 16KISK001K) and the SCADS.AI centre is gratefully acknowledged.
}

\begin{abstract}
    The classical Liouville property says that all bounded harmonic functions in $\mathds{R}^n$, i.e.\ all bounded functions satisfying $\Delta f = 0$, are constant. In this paper we obtain necessary and sufficient conditions on the symbol of a Fourier multiplier operator $m(D)$, such that the solutions $f$ to $m(D)f=0$ are Lebesgue a.e.\ constant (if $f$ is bounded) or coincide Lebesgue a.e.\ with a polynomial (if $f$ is polynomially bounded). The class of Fourier multipliers includes the (in general non-local) generators of L\'evy processes. For generators of L\'evy processes, we obtain necessary and sufficient conditions for a strong Liouville theorem where  $f$ is positive and grows at most exponentially fast. As an application of our results above we prove a coupling result for space-time L\'evy processes.
\end{abstract}

\subjclass[2010]{\emph{Primary:} 42B15; 35B53. \emph{Secondary:} 31C05, 35B10, 47G30, 60G51.} 

\keywords{Fourier multiplier; Liouville property; strong Liouville property; harmonic function; space-time harmonic function; L\'evy process; subordination; coupling}

\maketitle

\noindent
The classical Liouville property for the Laplace operator $\Delta := \sum_{k=1}^n \frac{\partial^2}{\partial x_k^2}$ on $\mathds{R}^n$ can be stated in the following way:
\begin{gather}\label{eqn:0.1}
    f\in L^\infty(\mathds{R}^n) \;\&\; \Delta f = 0 \implies f\equiv c\quad\text{almost everywhere}.
\end{gather}
Of course, $\Delta f$ has to be interpreted in a suitable sense since $f$ lacks regularity for a pointwise interpretation. Either one uses a mollifying argument based on convolution along with the fact that convolution and $\Delta$ commute or, as we do here, understands $\Delta f$ as a Schwartz distribution, i.e.\ $\Delta f$ is defined as the continuous linear form $\phi\mapsto\scalp{f}{\Delta\phi}$ for any $\phi\in C_c^\infty(\mathds{R}^n)$ (the smooth functions with compact support) or $\phi\in \Scal(\mathds{R}^n)$ (the rapidly decreasing smooth functions).

Recently, Alibaud \emph{et al.}~\cite{Ali-et-al20} and, independently, two of us \cite{BS21} gave proofs providing necessary and sufficient conditions ensuring an analogue of the Liouville property \eqref{eqn:0.1} for infinitesimal generators of L\'evy processes. The proof in \cite{Ali-et-al20} combines harmonic analysis and further methods from group theory, while the approach in \cite{BS21} uses mainly probabilistic arguments; the latter proof also yields the strong Liouville property where, in the appropriate analogue to \eqref{eqn:0.1}, $f\in L^\infty(\mathds{R}^n)$ is relaxed to $f\geq 0$ and at most exponential growth at infinity. Sufficient conditions for a \enquote{polynomial} Liouville property (if $f$ is polynomially bounded, then $f$ coincides a.e.\ with a polynomial) are due to K\"{u}hn \cite{K21}.

In the present paper, we give a very short and purely analytic proof for both the Liouville property and the polynomial Liouville property for L\'evy generators and -- as it turns out -- a much larger class of Fourier multiplier operators. In fact, the necessary and sufficient condition for the Liouville property is that $\xi=0$ is the only zero of the multiplier $m(\xi)$. For generators of L\'evy processes we refine the strong Liouville result proved in \cite{BS21} and we establish a further probabilistic interpretation of the Liouville property for L\'evy generators in terms of coupling and space-time harmonic functions.

\begin{notation}
Most of our notation is standard or self-explanatory. We write $\Fcal \phi(\xi) = \widehat\phi(\xi) = (2\pi)^{-n}\int_\mathds{R}^n e^{-ix\cdot\xi}\phi(x)\,dx$ and $\Fcal^{-1} u(x) = \int_{\mathds{R}^n} e^{ix\cdot\xi} u(\xi)\,d\xi$ for the Fourier and the inverse Fourier transform. We denote by $(T_h \phi)(x) := \phi(x - h)$ the shift by $h\in\mathds{R}^n$, $\widetilde\phi(x) := \phi(-x)$ is the reflection at the origin, and $\Lambda(x) := (1+|x|^2)^{1/2}$ is the standard weight function.
\end{notation}

\section{The proof of the Liouville theorem from \cite{BS21} revisited}\label{sec-1}

We will need a few notions from probability theory and stochastic processes, which can be found in Sato \cite{sato} or Jacob \cite{jac-1}, \cite{jac-2} and \cite{jac-3}, but the essential ingredient is the structure of the infinitesimal generator, see below. A \textbf{L\'evy process} is a stochastic process $(X_t)_{t\geq 0}$ with values in $\mathds{R}^n$ and sample paths $t\mapsto X_t(\omega)$ which are for almost all $\omega$ right-continuous with finite left-hand limits; moreover the random variables $X_{t_k}-X_{t_{k-1}}$, $0=t_0<t_1<\dots<t_m$, $m\in\mathds{N}$, are stochastically independent (independent increments) and each increment $X_{t_k}-X_{t_{k-1}}$ has the same distribution as $X_{t_k-t_{k-1}}$ (stationary increments). The fact that we are looking at a process with independent and stationary increments means that the distribution of $X_t$ (for any fixed $t>0$) characterizes the whole process; moreover, $X_t$ is necessarily infinitely divisible, so that its characteristic function (inverse Fourier transform) is of the form
\begin{gather}\label{eq:1.1}
    \Ee e^{i\xi\cdot X_t} = e^{-t\psi(\xi)},\quad\xi\in\mathds{R}^n,\; t>0,
\end{gather}
where the \textbf{characteristic exponent} $\psi(\xi)$ is uniquely given by the L\'evy--Khintchine formula
\begin{gather}\label{eq:1.2}
    \psi(\xi)
    = -ib\cdot\xi + \frac 12 \xi\cdot Q\xi + \int\limits_{0<|y|<1} \left(1-e^{iy\cdot\xi} + iy\cdot\xi\right) \nu(dy)  + \int\limits_{|y|\geq 1} \left(1-e^{iy\cdot\xi}\right) \nu(dy).
\end{gather}
The so-called L\'evy triplet $(b,Q,\nu)$ comprising a vector $b\in\mathds{R}^n$, a symmetric, positive semidefinite matrix $Q\in\mathds{R}^{n\times n}$, and a measure $\nu$ such that $\int_{y\neq 0} \min\left\{1,|y|^2\right\}\nu(dy)<\infty$, characterizes $\psi$ uniquely. For example, taking $b=0, Q = \id, \nu\equiv 0$ we get Brownian motion with its characteristic exponent $\psi(\xi) = \frac 12 |\xi|^2$, and the choice $b=0, Q=0, \nu(dy) = c_\alpha |y|^{-n-\alpha}\,dy$ with a suitable constant $c_\alpha$
yields a rotationally symmetric $\alpha$-stable process with characteristic exponent $|\xi|^\alpha$, $0<\alpha<2$.

Since L\'evy processes are Markov processes, their transition behaviour can be described by a transition semigroup $\Pcal_t f(x) = \Ee f(X_t+x)$ which, in turn, is uniquely characterized by the infinitesimal generator $\Lcal f := \frac d{dt}\Pcal_t f\big|_{t=0}$ (in the Banach space $C_\infty(\mathds{R}^n)$ of all continuous functions vanishing at infinity, say). Now the key point is the following observation:
\begin{fact}
    The infinitesimal generator $\Lcal = \Lcal_\psi$ of a L\'evy process with characteristic exponent $\psi(\xi)$ is on $C_c^\infty(\mathds{R}^n)$ a \textbf{Fourier multiplier operator} with symbol $-\psi(\xi)$, i.e.\
    \begin{gather}\label{eq:1.3}
        \Lcal_\psi\phi(x) = -\psi(D)\phi(x) = \Fcal_{\xi\to x}^{-1}\left(-\psi(\xi)\widehat\phi(\xi)\right),\quad\phi\in C_c(\mathds{R}^n),
    \end{gather}
and, combining this with \eqref{eq:1.2},
    \begin{align}\label{eq:1.3a}
    & \Lcal_\psi\phi(x) =  \ b\cdot\nabla \phi(x) + \frac12\, \nabla\cdot Q\nabla \phi(x)  \nonumber \\
    & \mbox{} + \int_{0<|y|<1} \left(\phi(x + y) - \phi(x) - y\cdot\nabla \phi(x)\right)\nu(dy)
    + \int_{|y|\geq 1} \left(\phi(x + y) - \phi(x)\right) \nu(dy).
    \end{align}
\end{fact}
Note that we get $\Lcal_\psi = \frac 12\Delta$, if $\psi(\xi) = \frac 12|\xi|^2$ (Brownian motion) and $\Lcal_\psi = -(-\Delta)^{\alpha/2}$, if $\psi(\xi) = |\xi|^\alpha$ (stable process).

Our first aim is to give a purely analytic proof of the following result.
\begin{theorem}[Liouville; \cite{Ali-et-al20}, \cite{BS21}]\label{th:1.1}
    Let $\psi$ be the characteristic exponent of a L\'evy process and denote by $\psi(D)$ the corresponding Fourier multiplier operator.
    Suppose $f \in L^\infty(\mathds{R}^n)$ is such that $\psi(D)f=0$ as a distribution, i.e.\
    \begin{gather}\label{eq:1.4}
        \scalp{f}{\widetilde\psi(D)\phi} = 0 \quad\text{for all\ } \phi \in C_c^\infty(\mathds{R}^n).
    \end{gather}
    If $\left\{\eta \in \mathds{R}^n \mid  \psi(\eta) = 0\right\} = \{0\}$, then $f \equiv \textrm{const}$ Lebesgue almost everywhere.

    Conversely, if $f \equiv \textrm{const}$ Lebesgue almost everywhere for every $f \in L^\infty(\mathds{R}^n)$ satisfying \eqref{eq:1.4},
    then $\left\{\eta \in \mathds{R}^n \mid  \psi(\eta) = 0\right\} = \{0\}$.
\end{theorem}

A formal  proof of the implication
\begin{align*}
    & \left\{\eta \in \mathds{R}^n \mid  \psi(\eta) = 0\right\} = \{0\} \\
    &  \quad\implies \ \  \text{every bounded solution of $\psi(D)f=0$  is a.e.\ constant}
\end{align*}
is very easy. Indeed, if $\psi(D)f = 0$, then $\psi(\eta) \widehat{f}(\eta) \equiv 0$. Hence
\begin{gather*}
    \supp \widehat{f} \subset \left\{\eta \in \mathds{R}^d \mid \ \psi(\eta) = 0\right\} = \{0\} .
\end{gather*}
Then $\widehat{f} = \sum_{|\alpha| \le N} c_\alpha \partial^\alpha \delta$\, for some $N \in \mathds{Z}_+$ and $c_\alpha \in \mathds{C}$. Hence $f$ is a polynomial. If it is bounded, it has to be constant. This proof can be made rigorous if $\psi$ is $C^\infty$-smooth (see \cite[Theorem 3.2]{BS21}). However, $\psi$ may be continuous but nowhere differentiable (see \cite[Remark 3.3]{BS21}), in which case defining the product of $\psi$ and the distribution $\widehat{f}$ is by no means trivial.

Our analytic proof of Theorem \ref {th:1.1} is based on the following standard result, which is known from the proof of Wiener's Tauberian theorem, cf.\ Rudin~\cite[Theorem~9.3]{Ru73}.
\begin{theorem}\label{th:1.2}
    If $f \in L^\infty(\mathds{R}^n)$, $Y$ is a linear subspace of $L^1(\mathds{R}^n)$, and
    \begin{gather}\label{eq:1.5}
        f \ast g = 0 \quad\text{for every\ } g \in Y ,
    \end{gather}
    then the set
    \begin{gather}\label{eq:1.6}
        Z(Y) := \bigcap_{g \in Y} \left\{\xi \in \mathds{R}^n \mid \widehat{g}(\xi) = 0\right\}
    \end{gather}
    contains the support of the tempered distribution $\widehat{f}$.
\end{theorem}
    Theorem \ref{th:1.2} says, in essence, that for $f\in L^\infty(\mathds{R}^n)$ and $g\in L^1(\mathds{R}^n)$ the condition $f\ast g=0$ implies that $\widehat{g}(\xi)=0$ for all $\xi$ in the support of $\widehat{f}$.
\begin{proof}[Proof of Liouville's Theorem~\ref{th:1.1}]
Since $\psi$ is the characteristic exponent of a L\'evy process, $\widetilde\psi(D)$ (as well as $\psi(D)$) is a continuous operator from $C_c^\infty(\mathds{R}^n)$ to $L^1(\mathds{R}^n)$, see e.g.\ \cite[Lemma 3.4]{S01} or the discussion in Example~\ref{ex:1.5} below. Thus, the dual pairing in \eqref{eq:1.4} is well-defined.

Suppose $\left\{\eta \in \mathds{R}^n \mid \psi(\eta) = 0\right\} = \{0\}$. Notice that $T_y \phi \in C_c^\infty(\mathds{R}^n)$ for every $\phi\in C_c^\infty(\mathds{R}^n)$ and all $y \in \mathds{R}^n$. Therefore, we see from \eqref{eq:1.4} that
\begin{gather*}
    \left(f \ast \widetilde{\widetilde\psi(D) \phi}\right)(y)
    = \scalp{f}{T_y\widetilde\psi(D) \phi}
    = \scalp{f}{\widetilde\psi(D) \left(T_y\phi\right)}
    = 0 \quad\text{for all\ } y \in \mathds{R}^n .
\end{gather*}
Hence
\begin{gather*}
    f \ast \widetilde{\widetilde\psi(D) \phi} = 0 \quad\text{for all\ } \phi \in C_c^\infty(\mathds{R}^n) .
\end{gather*}
It is easy to see that
\begin{align*}
    \bigcap_{\phi \in C_c^\infty(\mathds{R}^n)} \bigg\{\eta \in \mathds{R}^n \mid \widehat{\widetilde{\widetilde\psi(D) \phi}}(\eta) = 0\bigg\}
    &= \bigcap_{\phi \in C_c^\infty(\mathds{R}^n)} \bigg\{\eta \in \mathds{R}^n \mid  \widehat{\widetilde\psi(D) \phi}(-\eta) = 0\bigg\} \\
    &= \bigcap_{\phi \in C_c^\infty(\mathds{R}^n)} \left\{\eta \in \mathds{R}^n \mid  \psi(\eta) \widehat{\phi}(-\eta) = 0\right\}\\
    &= \left\{\eta \in \mathds{R}^n \mid \psi(\eta) = 0\right\} = \{0\} .
\end{align*}
Applying Wiener's theorem (Theorem~\ref{th:1.2}) with
\begin{gather*}
    Y := \left\{\widetilde{\widetilde\psi(D) \phi}\ \big| \ \phi \in C_c^\infty(\mathds{R}^n)\right\} \subset L^1(\mathds{R}^n) ,
\end{gather*}
we conclude that the support of the tempered distribution $\widehat{f}$ is contained in $\{0\}$. Therefore, there exist $N \in \mathds{N}_0$ and $c_\alpha \in \mathds{C}$ ($\alpha \in \mathds{N}_0^n$, $|\alpha| \leq N$) such that
\begin{gather*}
    \widehat{f} = \sum_{|\alpha| \leq N} c_\alpha \partial^\alpha \delta_0,
\end{gather*}
see, e.g.\ \cite[Thorems 6.24 and 6.25]{Ru73}. Inverting the Fourier transform, and using the assumption that $f$ is bounded, shows that $f\equiv c_0$ Lebesgue a.e.

For the converse direction, we assume the contrary, i.e.\ that the zero-set $\{\eta\in\mathds{R}^n \mid \psi(\eta)=0\}$ contains some $\gamma\neq 0$. Then $f(x) := e^{i\gamma\cdot x}$ satisfies $\psi(D_{ x})e^{i\gamma\cdot x} = e^{i\gamma\cdot x}\psi(\gamma) = 0$ for all $x\in\mathds{R}^n$. Thus, $f$ is a non-constant solution, and we are done. We note in passing that since $\psi(-\gamma)=\overline{\psi(\gamma)}$, $-\gamma \in\{ \psi=0\}$, and we can even get a real-valued solution:
\begin{equation*}
    2\psi(D_x)\cos(\gamma\cdot x)
    =\psi(D_x)\left(e^{i\gamma\cdot x}+e^{-i\gamma\cdot x}\right)
    =e^{i\gamma\cdot x}\psi(\gamma)+e^{-i\gamma\cdot x} \psi(-\gamma)
    =0.
\qedhere
\end{equation*}
\end{proof}

\begin{remark}\label{re:1.3}
    With a bit more effort, see \cite{BS21} or \cite{Ali-et-al20}, one can show that all bounded solutions in the converse direction of Theorem~\ref{th:1.1} are necessarily periodic. Since $\{\psi = 0\}$ is a closed subgroup of the additive group $(\mathds{R}^n,+)$, the periodicity group of all bounded solutions is given by the dual lattice $\{\psi = 0\}^{\boxperp} := \{x\in\mathds{R}^n \mid e^{i \gamma\cdot x}=1,\, \forall \gamma\in\{\psi=0\}\}$. The proof in \cite{Ali-et-al20} actually shows that a L\'evy generator $\psi(D)$ has the Liouville property if, and only if, $\{\psi = 0\}^{\boxperp} = \mathds{R}^n$.
\end{remark}

The above proof of Theorem~\ref{th:1.1} extends without change to Fourier multiplier operators that map $C_c^\infty(\mathds{R}^n)$ into $L^1(\mathds{R}^n)$.
\begin{theorem}\label{co:1.4}
    Let $m\in C(\mathds{R}^n)$ be such the Fourier multiplier operator
    \begin{gather*}
        C_c^\infty(\mathds{R}^n)\ni \phi \mapsto \widetilde m(D)\phi := \Fcal^{-1}(\widetilde m\widehat\phi)
    \end{gather*}
    maps $C_c^\infty(\mathds{R}^n)$ into $L^1(\mathds{R}^n)$. Suppose $f \in L^\infty(\mathds{R}^n)$ is such that $m(D)f=0$ as a distribution, i.e.\
    \begin{gather}\label{eq:1.4'}
        \scalp{f}{\widetilde{m}(D)\phi} = 0 \quad\text{for all\ } \phi \in C_c^\infty(\mathds{R}^n).
    \end{gather}
    If $\left\{\eta \in \mathds{R}^n \mid  m(\eta) = 0\right\} \subset \{0\}$, then $f \equiv \textrm{const}$ Lebesgue almost everywhere.

    Conversely, if $f \equiv \textrm{const}$ Lebesgue almost everywhere for every complex-valued  $f \in L^\infty(\mathds{R}^n)$ satisfying \eqref{eq:1.4'},
    then $\left\{\eta \in \mathds{R}^n \mid  m(\eta) = 0\right\} \subset \{0\}$. If $m(\eta)=0$ implies that $m(-\eta)=0$, then it is enough to consider real-valued $f\in L^\infty(\mathds{R}^n)$ satisfying \eqref{eq:1.4'}.
\end{theorem}
\begin{remark}\label{re:1.5}
If  $\left\{\eta \in \mathds{R}^n \mid  m(\eta) = 0\right\} \subsetneqq \{0\}$, i.e.\ $\left\{\eta \in \mathds{R}^n \mid  m(\eta) = 0\right\} = \emptyset$,
then $f \equiv 0$ is the only constant solution of $m(D)f=0$.
\end{remark}

\begin{example}\label{ex:1.5}
Let us give a few examples of multipliers satisfying the key assumption of Theorem~\ref{co:1.4}. The following multipliers $\kappa$ are such that $\kappa(D)$ maps $C_c^\infty(\mathds{R}^n)$ into $L^1(\mathds{R}^n)$.
\begin{enumerate}\itemsep=5pt
\item\label{ex:1.5-1}
$\kappa$ is a linear combination of terms of the form $ab$, where $a$ is the Fourier transform of a finite Borel measure on $\mathds{R}^n$, and all partial derivatives
$\partial^\alpha_\xi b$ with $|\alpha| \le n + 1$ are polynomially bounded.

\emph{Indeed:} Let $b_N := b\Lambda^{-2N}$, $N \in \mathds{R}^n$. For a sufficiently large $N$, all partial derivatives $\partial^\alpha_\xi b_N$ with $|\alpha| \le n + 1$ belong to $L^1(\mathds{R}^n)$. Then $x^\alpha \Fcal^{-1}\left(b_N\right) \in L^\infty(\mathds{R}^n)$, $|\alpha| \le n + 1$. Hence
\begin{align*}
    & \Fcal^{-1}\left(b_N\right), \ |x_j|^{n + 1}  \Fcal^{-1}\left(b_N\right) \in L^\infty(\mathds{R}^n) , \ j = 1, \dots, n \\
    &\implies (1 + |x|)^{n + 1}\Fcal^{-1}\left(b_N\right) \in L^\infty(\mathds{R}^n)
    \implies \Fcal^{-1}\left(b_N\right) \in L^1(\mathds{R}^n) .
\end{align*}

Let $\phi\in C_c^\infty(\mathds{R}^n) \subset L^1(\mathds{R}^n)$. We have $b_N(D)\phi = (2\pi)^{-n} \Fcal^{-1}(b_N)*\phi$, and  it follows from Young's inequality
\begin{gather*}
    \|(\Fcal^{-1}b_N)*\phi\|_{L^1}
    \leq \|\Fcal^{-1}b_N\|_{L^1} \|\phi\|_{L^1}
\end{gather*}
that $b_N(D)$ maps $C_c^\infty(\mathds{R}^n)$ into $L^1(\mathds{R}^n)$. Since the differential operator $\Lambda^{2N}(D)$ maps $C_c^\infty(\mathds{R}^n)$ into itself, and $(ab)(D) = a(D)b_N(D)\Lambda^{2N}(D)$, it is left to show that $a(D)$ maps $L^1(\mathds{R}^n)$ into itself. Since $a = \Fcal\mu$ for a finite Borel measure $\mu$,
\begin{gather*}
    (a(D)g)(x) = (2\pi)^{-n} \int_{\mathds{R}^n} g(x - y)\, \mu(dy), \quad x \in \mathds{R}^n , \qquad g \in \mathds{R}^n ,
\end{gather*}
and
\begin{gather*}
    \|a(D)g\|_{L^1}
    \leq (2\pi)^{-n}  \mu(\mathds{R}^n) \|g\|_{L^1} \quad\text{for all}\quad g \in L^1(\mathds{R}^n) .
\end{gather*}

A particular example is the characteristic exponent of a L\'evy process
\begin{align*}
   \psi(\xi)
    &= -ib\cdot\xi + \frac 12 \xi\cdot Q\xi + \int\limits_{0<|y|<1} \left(1-e^{iy\cdot\xi} + iy\cdot\xi\right) \nu(dy)  + \int\limits_{|y|\geq 1} \left(1-e^{iy\cdot\xi}\right) \nu(dy) \\
    &=: \psi_0(\xi) - \int\limits_{|y|\geq 1} e^{iy\cdot\xi}\,\nu(dy)  .
\end{align*}
The last term in the above formula is the Fourier transform of a finite Borel measure.
The smoothness of $\psi_0$ follows immediately from the differentiation lemma for pa\-ra\-me\-ter-\/dependent integrals and the observation that the integrand of the formally differentiated function $\partial^\alpha_{\xi} \psi_0$ can be bounded by $\mathrm{const}\, |y|^2$, which is $\nu$-integrable over $\{y \in \mathds{R}^n : \ 0<|y|<1\}$.
This bound also shows that
    $|\partial^\alpha\psi_0(\xi)|\leq c_0(1+|\xi|^2)$ if $|\alpha|=0$,
    $|\partial^\alpha\psi_0(\xi)|\leq c_1(1+|\xi|)$ if $|\alpha|=1$, and
    $|\partial^\alpha\psi_0(\xi)|\leq c_\ell$ if $|\alpha|=\ell\geq 2$,
thus we get polynomial boundedness of $\psi_0$ and its derivatives. A fully worked-out proof can be found in \cite[Lemma~4]{BS21} as well as \cite[Lemma 3.6.22, Theorem 3.7.13]{jac-1}.

\item\label{ex:1.5-5}
    Any $\kappa$ of the form
    \begin{gather*}
        \kappa(\xi)
        = \sum_{|\alpha|=0}^{2s} c_\alpha \frac{i^{|\alpha|}}{\alpha!} \xi^\alpha
        + \int\limits_{0<|y|<1}\bigg[1-e^{iy\cdot\xi}+\sum_{|\alpha|=0}^{2s-1}\frac{i^{|\alpha|}}{\alpha!} y^\alpha\xi^\alpha\bigg]\,\nu(dy)
        + \int\limits_{|y|\geq 1} \left(1-e^{iy\cdot\xi}\right)\nu(dy)
    \end{gather*}
    with $s\in\mathds{N}$, $c_\alpha\in\mathds{R}$, and a measure $\nu$ on $\mathds{R}^n\setminus\{0\}$ such that $\int_{y\neq 0} \min\left\{1,|y|^{2s}\right\}\,\nu(ds)<\infty$.
    (As usual, for any $\alpha\in\mathds{N}_0^n$ and $\xi \in\mathds{R}^n$, we define $\alpha!:= \prod_1^n \alpha_k!$ and $\xi^\alpha := \prod_1^n \xi_k^{\alpha_k}$.)
    The proof of this assertion goes along the lines of Part~\ref{ex:1.5-1}.

    Functions of this type appear naturally in positivity questions related to generalized functions, see e.g.\ Gelfand \& Vilenkin \cite[Chapter II.4]{gel-vil64} or Wendland \cite{wendland}. Some authors call the function $-\kappa$ (under suitable additional conditions on $c_\alpha$'s) a \textbf{conditionally positive definite function}. Note that $s=1$ is just the L\'evy--Khintchine formula \eqref{eq:1.2}.
\end{enumerate}
\end{example}

\section{The Liouville theorem for polynomially bounded functions}\label{PolyBdd}
We are now going to show that our argument used in the proof of Theorems~\ref{th:1.1} and \ref{co:1.4} extends to polynomially bounded
functions $f$ in \eqref{eq:1.4}.

To simplify the presentation, we use the function $\Lambda(x) = \left(1+|x|^2\right)^{1/2}$ as well as the following function spaces.
Let $\beta \geq 0$, and
\begin{align*}
    L^1_\beta(\mathds{R}^n)
    &:= \left\{g \in L^1(\mathds{R}^n) \mid \|g\|_{L^1_\beta} := \int_{\mathds{R}^n} |g(x)| \Lambda(x)^\beta\, dx < \infty\right\},
\\
    L^\infty_{-\beta}(\mathds{R}^n)
    &:= \left\{f \in L^\infty_{\mathrm{loc}}(\mathds{R}^n) \mid  \|f\|_{L^\infty_{-\beta}} := \left\|\Lambda^{-\beta} f\right\|_{L^\infty} = \esssup_{x \in \mathds{R}^n}  \Lambda(x)^{-\beta} |f(x)| < \infty\right\} .
\end{align*}
Obviously, $L^\infty_{-\beta}(\mathds{R}^n) \subset \Scal'(\mathds{R}^n)$, $\Scal(\mathds{R}^n) \subset L^1_\beta(\mathds{R}^n) \subset L^1(\mathds{R}^n)$, and $L^1_\beta(\mathds{R}^n)$ is a convolution algebra; $A_\beta := \left\{c\delta_0 + g \mid c\in\mathds{C}, g\in L^1_\beta(\mathds{R}^n)\right\}$ is $L^1_\beta(\mathds{R}^n)$ with a unit attached, cf.\ Rudin \cite[10.3(d), 11.13(e)]{Ru73}.

We need the following analogue of Theorem~\ref{th:1.2} for the pair $(L^1_\beta(\mathds{R}^n),L^\infty_{-\beta}(\mathds{R}^n))$.
\begin{theorem}\label{th:2.1}
    If $f \in L^\infty_{-\beta}(\mathds{R}^n)$, $Y$ is a linear subspace of $L^1_\beta(\mathds{R}^n)$, and
    \begin{gather}\label{eq:2.1}
        f \ast g = 0 \quad\text{for every\ } g \in Y ,
    \end{gather}
    then the set
    \begin{gather}\label{eq:2.2}
        Z(Y) := \bigcap_{g \in Y} \left\{\xi \in \mathds{R}^n \mid \widehat{g}(\xi) = 0\right\}
    \end{gather}
    contains the support of the tempered distribution $\widehat{f}$.
\end{theorem}
\begin{proof}
    Pick any $\xi_0 \in \mathds{R}^n\setminus Z(Y)$. There exists some $g \in Y$ such that $\widehat{g}(\xi_0) = 1$. Since $\widehat{g}$ is continuous, there is a neighbourhood $V=V(\xi_0)$ of $\xi_0$ such that $\left|\widehat{g}(\xi) - 1\right| < 1/2$ for all $\xi \in V$.

    To prove the theorem, it is sufficient to show that $\widehat{f} = 0$ in $V$, or, equivalently, that $\scalp{\widehat{f}}{\widehat{v}} = 0$ for every $v \in \Scal(\mathds{R}^n)$ whose Fourier transform $\widehat{v}$ has its support in $V$. Since
\begin{gather*}
    \scalp{\widehat{f}}{\widehat{v}}
    =  (2\pi)^{-n} \scalp{f}{\widetilde{v}}
    = (2\pi)^{-n} (f\ast v)(0),
\end{gather*}
    it is sufficient to prove that $f\ast v(0) = 0$.

    Take $\phi \in C_c^\infty(V)$ such that $0 \leq \phi \leq 1$, and $\phi(\xi) = 1$ for every $\xi$ in the support of $\widehat{v}$. Since $\Fcal^{-1}{\phi} \in \Scal(\mathds{R}^n)$, there exists an element $u \in A_\beta$ such that
\begin{gather*}
    \widehat{u} = \phi\, \widehat{g} + 1 - \phi .
\end{gather*}
    It is easy to see that $\re \widehat{u}(\xi) > 1/2$ for all $\xi \in \mathds{R}^n$. Then there exists some $w \in A_\beta$ such that $\widehat{w} = 1/\widehat{u}$, see, e.g.\ \cite[Theorems 1.41 and 2.11]{D56} or \cite[Theorem 1.3]{FGL06}. Hence
\begin{gather*}
    \widehat{v}
    = \widehat{w}\, \widehat{u}\, \widehat{v} = \widehat{w}\, \phi\, \widehat{g}\, \widehat{v}
    =  \widehat{g}\, \widehat{w}\, \phi\, \widehat{v} .
\end{gather*}
    So, $v = g\ast G$ for some $G \in L^1_\beta(\mathds{R}^n)$.\footnote{The equality $v = g\ast G$ with $G \in L^1(\mathds{R}^n)$ is derived in the proof of \cite[Theorem 9.3]{Ru73}. Unfortunately, this version does not seem sufficient for the proof of the equality $f\ast (g\ast G) = (f\ast g)\ast G$ when $f \in L^\infty_{-\beta}(\mathds{R}^n)$.} Iterating the standard Peetre inequality $\Lambda(x-y)\leq \sqrt{2}\Lambda(x)\Lambda(y)$, we see that
\begin{gather*}
    \Lambda(x-y)^\beta
    \leq 2^\beta \Lambda(x)^\beta \Lambda(y-z)^\beta \Lambda(z)^\beta
    \quad\text{for all\ } x, y, z \in \mathds{R}^n.
\end{gather*}
This yields
\begin{align*}
     \int_{\mathds{R}^n}& |f(x - y)| \left(\int_{\mathds{R}^n} |g(y - z)| |G(z)|\, dz\right)dy \\
    & \leq 2^\beta \Lambda(x)^\beta \int_{\mathds{R}^n} |f(x - y)| \Lambda(x - y)^{-\beta}
    \left(\int_{\mathds{R}^n} |g(y - z)| \Lambda(y-z)^\beta \Lambda(z)^\beta |G(z)|\, dz\right)dy \\
    & \leq 2^\beta \Lambda(x)^\beta \|f\|_{L^\infty_{-\beta}} \|g\|_{L^1_\beta} \|G\|_{L^1_\beta} < \infty ,
\end{align*}
and the Fubini--Tonelli theorem implies $f\ast (g\ast G) = (f\ast g)\ast G$. Finally,
\begin{gather*}
    f\ast v = f\ast (g\ast G) = (f\ast g)\ast G = 0\ast G = 0 .
\qedhere
\end{gather*}
\end{proof}

We can now state and prove the Liouville theorem for polynomially bounded functions.
\begin{theorem}[Liouville property for polynomially bounded functions]\label{th:2.2}
    Let $m\in C(\mathds{R}^n)$ be such that the Fourier multiplier operator
    \begin{gather*}
        C_c^\infty(\mathds{R}^n)\ni \phi \mapsto \widetilde m(D)\phi := \Fcal^{-1}(\widetilde m\widehat\phi)
    \end{gather*}
    maps $C_c^\infty(\mathds{R}^n)$ into $L^1_{\beta}(\mathds{R}^n)$. Suppose $f \in L^\infty_{-\beta}(\mathds{R}^n)$ is such that $m(D)f=0$ as a distribution, i.e.\
     \begin{gather}\label{eq:2.3}
        \scalp{f}{\widetilde m(D)\phi} = 0 \quad\text{for all\ } \phi \in C_c^\infty(\mathds{R}^n).
    \end{gather}
    If $\left\{\eta \in \mathds{R}^n \mid  m(\eta) = 0\right\} \subset \{0\}$, then $f$ coincides Lebesgue a.e.\ with a polynomial of degree at most $\entier{\beta}$.

    Conversely, if every complex-valued $f \in L^\infty_{-\beta}(\mathds{R}^n)$ satisfying \eqref{eq:2.3} coincides Lebesgue a.e.\ with a polynomial, then $\left\{\eta \in \mathds{R}^n \mid  m(\eta) = 0\right\} \subset \{0\}$. If $m(\eta)=0$ implies that $m(-\eta)=0$, then it is enough to consider real-valued $f \in L^\infty_{-\beta}(\mathds{R}^n)$ satisfying \eqref{eq:2.3}.
\end{theorem}
\begin{proof}
    If we replace in the proof of Theorem~\ref{th:1.2}\, $\psi\rightsquigarrow m$, $L^1(\mathds{R}^n)\rightsquigarrow L^1_\beta(\mathds{R}^n)$ and $L^\infty(\mathds{R}^n) \rightsquigarrow L^\infty_{-\beta}(\mathds{R}^n)$, we can follow the argument line-by line up to the point where we get
    \begin{gather*}
        \widehat{f} = \sum_{|\alpha| \leq N} c_\alpha \partial^\alpha \delta_0.
    \end{gather*}
    Again, we invert the Fourier transform and use the polynomial boundedness of $f$ to see that $f$ coincides Lebesgue a.e.\ with a polynomial of degree less or equal
    than $\entier{\beta}$. The converse statement follows from that in Theorem \ref{co:1.4}. Indeed, if every $f \in L^\infty_{-\beta}(\mathds{R}^n)$ satisfying \eqref{eq:2.3}
    coincides Lebesgue a.e.\ with a polynomial, then the same is true for any such $f$ in $L^\infty(\mathds{R}^n) \subseteq L^\infty_{-\beta}(\mathds{R}^n)$.
    Since the only polynomials contained in $L^\infty(\mathds{R}^n)$ are constants, one can apply the converse statement in Theorem \ref{co:1.4}.
\end{proof}

\begin{example}\label{ex:2.3}
    In this example we discuss the multipliers from Example \ref{ex:1.5} in the setting of Theorem \ref{th:2.2}. The following conditions ensure that $\kappa(D)$
    maps $C_c^\infty(\mathds{R}^n)$ into $L_\beta^1(\mathds{R}^n)$.
\begin{enumerate}\itemsep5pt
\item\label{ex:2.3-1}
    $\kappa$ is a linear combination of terms of the form $ab$, where $a = \Fcal\mu$,\ $\mu$ is a finite Borel measure on $\mathds{R}^n$ such that $\int \Lambda^\beta(y)\mu(dy)<\infty$, and all partial derivatives $\partial^\alpha_\xi b$ with $|\alpha| \le \entier{n + \beta} + 1$ are polynomially bounded.

    \emph{Indeed:} Let $b_N := b\Lambda^{-2N}$, $N \in \mathds{R}^n$. For a sufficiently large $N$, all partial derivatives $\partial^\alpha_\xi b_N$ with $|\alpha| \le \entier{n + \beta} + 1$ belong to $L^1(\mathds{R}^n)$. Then $x^\alpha \Fcal^{-1}\left(b_N\right) \in L^\infty(\mathds{R}^n)$, $|\alpha| \le \entier{n + \beta} + 1$. Hence
    \begin{align*}
        & \Fcal^{-1}\left(b_N\right), \ |x_j|^{\entier{n + \beta} + 1}  \Fcal^{-1}\left(b_N\right) \in L^\infty(\mathds{R}^n) , \ j = 1, \dots, n \\
        & \implies (1 + |x|)^{\entier{n + \beta} + 1}\Fcal^{-1}\left(b_N\right) \in L^\infty(\mathds{R}^n)
        \implies \Fcal^{-1}\left(b_N\right) \in L^1_\beta(\mathds{R}^n) .
    \end{align*}

    Let $\phi\in C_c^\infty(\mathds{R}^n) \subset L^1_\beta(\mathds{R}^n)$. We have $b_N(D)\phi = (2\pi)^{-n} \Fcal^{-1}(b_N)*\phi$, and  it follows from Young's and Peetre's inequalities
    \begin{gather*}
        \|(\Fcal^{-1}b_N)*\phi\|_{L^1_\beta} = \left\|\Lambda^\beta\big((\Fcal^{-1}b_N)*\phi\big)\right\|_{L^1}
        \leq 2^{\beta/2} \left\|(\Lambda^\beta\Fcal^{-1}b_N)*(\Lambda^\beta\phi)\right\|_{L^1} \\
         \leq 2^{\beta/2} \|\Lambda^\beta\Fcal^{-1}b_N\|_{L^1} \|\Lambda^\beta\phi\|_{L^1}
        = 2^{\beta/2} \|\Fcal^{-1}b_N\|_{L^1_\beta} \|\phi\|_{L^1_\beta}
    \end{gather*}
   that $b_N(D)$ maps $C_c^\infty(\mathds{R}^n)$ into $L^1_\beta(\mathds{R}^n)$. Since the differential operator $\Lambda^{2N}(D)$ maps $C_c^\infty(\mathds{R}^n)$ into itself, and $(ab)(D) = a(D)b_N(D)\Lambda^{2N}(D)$, it is left to show that $a(D)$ maps $L^1_\beta(\mathds{R}^n)$ into itself. Using Peetre's inequality again, we deduce from
    \begin{gather*}
        (a(D)g)(x) = (2\pi)^{-n} \int_{\mathds{R}^n} g(x - y)\, \mu(dy), \quad x \in \mathds{R}^n, \; g \in \mathds{R}^n ,
    \end{gather*}
    that
    \begin{align*}
        &(2\pi)^n\|a(D)g\|_{L^1_\beta(\mathds{R}^n)}\\
        &\leq \int_\mathds{R}^n \Lambda^\beta(x)\int_\mathds{R}^n |g(x - y)|\,\mu(dy)\,dx
        \leq 2^{\beta/2} \int_\mathds{R}^n |g(z)|\Lambda^\beta(z)\,dz \int_\mathds{R}^n\Lambda^\beta(y)\,\mu(dy)  \\
        &= 2^{\beta/2} \int_\mathds{R}^n\Lambda^\beta(y)\,\mu(dy)\,  \|g\|_{L^1_\beta}\,.
    \end{align*}

    With a bit more effort, one can actually show that $\int \Lambda^\beta(y)\,\mu(dy) < \infty$ is also a necessary condition for $a(D)$ to map $C_c^\infty(\mathds{R}^n)$ into $L_\beta^1(\mathds{R}^n)$ (see \cite[Theorem 3]{BKS21} for a proof).

A particular example is the characteristic exponent $\kappa=\psi$ of a L\'evy process such that the L\'evy measure has finite moments of order $\beta$
(cf.\ Example \ref{ex:1.5}\ref{ex:1.5-1}).

\item\label{ex:2.3-5}
    Any $\kappa$ of the form
    \begin{gather*}
        \kappa(\xi)
        = \sum_{|\alpha|=0}^{2s} c_\alpha \frac{i^{|\alpha|}}{\alpha!} \xi^\alpha
        + \int_{0<|y|<1}\bigg[1-e^{iy\cdot\xi}+\sum_{|\alpha|=0}^{2s-1}\frac{i^{|\alpha|}}{\alpha!} y^\alpha\xi^\alpha\bigg]\,\nu(dy)
        + \int_{|y|\geq 1} \left(1-e^{iy\cdot\xi}\right)\nu(dy)
    \end{gather*}
     with $s\in\mathds{N}$, $c_\alpha\in\mathds{R}$, and a measure $\nu$ on $\mathds{R}^n\setminus\{0\}$ such that $\int_{0<|y|<1} |y|^{2s}\,\nu(dy) + \int_{|y|\geq 1}|y|^\beta\,\nu(dy)<\infty$.
\end{enumerate}
\end{example}

It is not difficult to extend Theorem \ref{th:2.2} from solutions of the equation $m(D) f = 0$ to solutions of $m(D) f = p$, where
$p$ is a polynomial.
\begin{corollary}\label{th:2.4}
 Let $m \in C(\mathds{R}^n)$ be as in Theorem~\ref{th:2.2}, and let $p$ be a polynomial. Suppose $f \in L^\infty_{-\beta}(\mathds{R}^n)$ is such that
    \begin{equation}\label{eq:2.4}
        \scalp{f}{\widetilde m(D) \phi} = \scalp{p}{\phi} \quad\text{for all\ } \phi \in C_c^\infty(\mathds{R}^n) .
    \end{equation}
    If $\left\{\eta \in \mathds{R}^n \mid  m(\eta) = 0\right\} \subset \{0\}$, then $f$ coincides Lebesgue a.e.\ with a polynomial of degree at most $\entier{\beta}$. Conversely,
    if there exists $f \in L^\infty_{-\beta}(\mathds{R}^n)$ satisfying \eqref{eq:2.4}, and every such $f$ coincides Lebesgue a.e.\ with a polynomial, then $\left\{\eta \in \rn \mid  m(\eta) = 0\right\} \subset \{0\}$.
\end{corollary}
\begin{proof}
    Pick $k \in \mathds{N}$ such that $2k$ is greater than the degree of $p$. Then $\Delta^k p = 0$. Set $m_k(\xi) := |\xi|^{2k} m(\xi)$; clearly, $m_k \in C(\mathds{R}^n)$. Since $\Delta^k$ maps $C_c^\infty(\mathds{R}^n)$ continuously into itself,  $\widetilde{m_k}(D) = \widetilde m(D) \Delta^k$ maps $C_c^\infty(\mathds{R}^n)$ into $L^1_\beta(\mathds{R}^n)$. Moreover,
    \begin{gather*}
        \scalp{f}{\widetilde m_k(D) \phi}
        = \scalp{f}{\widetilde m(D)(\Delta^k\phi)}
        = \scalp{p}{\Delta^k\phi}
        = \scalp{\Delta^k p}{\phi}
        = \scalp{0}{\phi}
        = 0
    \end{gather*}
    for all $\phi \in C_c^\infty(\mathds{R}^n)$. Now we can apply Theorem \ref{th:2.2} with $m_k$ in place of $m$. Note that $\{m_k=0\}=\{m=0\}\cup \{0\}$, i.e.\ $\{m_k=0\}\subset \{0\}$ if, and only if, $\{m=0\}\subset\{0\}$.
\end{proof}

The following result shows that the polynomial $f$ appearing in Theorem \ref{th:2.2} has degree at most 1 in the case where $n = 1$ and
$m$ is the characteristic exponent of a L\'evy process.

\begin{corollary}\label{th:4.14}
    Let $\psi : \mathds{R} \to \mathds{C}$\ be the characteristic exponent of a L\'evy process with L\'evy triplet $(Q, b, \nu)$. Suppose there exists a $\beta \ge 0$ such that $\int_{|y| \ge 1} |y|^\beta \,\nu(dy) < \infty$. If $\left\{\xi \in \mathds{R} \mid \ \psi(\xi) = 0\right\} = \{0\}$ and $f \in L^\infty_{-\beta}(\mathds{R})$ is a weak solution of the equation $\psi(D) f = 0$, then $f(x) = c_1x + c_0$ with some constants $c_1$ and $c_0$.
\end{corollary}

\begin{proof}
It follows from Theorem \ref{th:2.2} that $f$ is a polynomial of degree at most $[\beta]$. According to Lemma 2.4 in \cite{KS21}, the degree of\, $f$\,
is less than or equal to $2$.
The lemma deals with solutions of the equation $\psi(D) f = \mathrm{const}$. In the case of $\psi(D) f = 0$, its proof (see, in particular, the last paragraph
of the proof) shows that  the degree of\, $f$\, is actually less than or equal to $1$.
\end{proof}

The above result does not hold in the multi-dimensional case $n \ge 2$.
Note that in the case $n = 1$, \ $f'' = 0 \ \implies \ f(x) = ax + b$, while in the case $n = 2$, \ $\Delta f = 0$ has polynomial solutions of any degree, e.g. $\re (x_1 + i x_2)^k$, $k \in \mathds{N}$.

\section{The Liouville theorem for slowly growing functions}\label{sec-slow}

Theorem~\ref{th:2.1} covers bounded functions $f$, while Theorem~\ref{th:2.2} is about functions whose growth is compared with the growth of a polynomial. This leaves a gap where $f$ grows slower than a polynomial, e.g.\ at a logarithmic scale. To deal with this case, we need the notion of a measurable, \textbf{locally bounded, submultiplicative function}, i.e.\ a locally bounded measurable function $h : \mathds{R}^n\to (0, \infty)$ satisfying
\begin{gather*}
    h(x+y) \leq c h(x)h(y) \quad\text{for some $c\geq 1$ and all\ } x,y \in \mathds{R}^n.
\end{gather*}
Without loss of generality, we will always assume that $h\geq 1$, otherwise we would replace $h(x)$ by $h(x)+1$. Typical examples of submultiplicate functions are
$(1+|x|)^{\beta}$, $\Lambda(x)^\beta$, $e^{\alpha |x|^\beta}$ for $\beta \in [0, 1]$, and $\log^{\beta}(|x|+e)$ for $\beta\ge 0$.
Observe that every submultiplicative function is exponentially bounded. For further details see \cite[Section~25]{sato} or \cite[Section~II.\S1]{KPS82}.

Fix some locally bounded submultiplicative $h$ that satisfies, in addition,
\begin{align}\label{eq:3.1}
    \lim_{|x|\to\infty} \Lambda^{-k}(x) h(x)=0
    \quad\text{for some\ } k\in\mathds{N}.
\end{align}
The condition \eqref{eq:3.1} implies the so-called \textbf{GRS (Gelfand--Raikov--Shilov)-condition}, see \cite{BK22} and \cite{FGL06} for details.
We replace the pair $(L_\beta^1(\mathds{R}^n),L^\infty_{-\beta}(\mathds{R}^n))$ by
\begin{align*}
&   L_h^1(\mathds{R}^n)
    :=\left\{g\in L^1(\mathds{R}^n) \mid \|g\|_{L^1_{h}}:=\int_{\mathds{R}^n} |g(x)|h(x)\,dx < \infty\right\},\\
&    L^\infty_{h^{-1}}(\mathds{R}^n)
    :=\left\{f\in L^\infty_{\mathrm{loc}}(\mathds{R}^n) \mid \|f\|_{L^\infty_{h^{-1}}} := \|h^{-1}f\|_{L^\infty} = \esssup_{x\in \mathds{R}^n}h(x)^{-1}|f(x)|<\infty \right\}
\intertext{and use, instead of $A_\beta$,}
    A_h &:=\{c\delta_0+ g \mid c\in\mathds{C},g\in L^1_{\widetilde h}(\mathds{R}^n)\}.
\end{align*}
The family $A_h$ is a convolution algebra and, due to the GRS-condition, an element of $u\in A_h$ is invertible if, and only if, it is invertible in $A_1$, where $A_1$ is the algebra with $h=1$.  One can replace the Peetre inequality by $h(x-y)\leq c h(x)h(-y)$, which is a direct consequence of the submultiplicativity of $h$, and then get by iteration $h(x-y)\leq c^2 h(x)h(-(y-z))h(-z)$. This allows one to repeat the arguments in the proof of Theorem \ref{th:2.1} and arrive at a version of this theorem with $L_\beta^1 \rightsquigarrow L_{\widetilde h}^1$ and $L_{-\beta}^1 \rightsquigarrow L_{h^{-1}}^\infty$. If $\widetilde m(D)$ maps $C_c^\infty(\mathds{R}^n)$ into $L^1_h(\mathds{R}^n)$, one can apply this version of Theorem \ref{th:2.1} with
\begin{gather*}
    Y := \left\{\widetilde{\widetilde m(D) \phi}\ \big| \ \phi \in C_c^\infty(\mathds{R}^n)\right\} \subset L_{\widetilde h}^1(\mathds{R}^n) .
\end{gather*}
This results in the following analogue of Theorem~\ref{th:2.2} for slowly growing functions.
\begin{theorem}[Liouville property for slowly growing functions]\label{th:3.1}
    Let $m\in C(\mathds{R}^n)$ be such that the Fourier multiplier operator
    \begin{gather*}
        C_c^\infty(\mathds{R}^n)\ni \phi \mapsto \widetilde m(D)\phi := \Fcal^{-1}(\widetilde m\widehat\phi)
    \end{gather*}
    maps $C_c^\infty(\mathds{R}^n)$ into $L^1_h(\mathds{R}^n)$. Suppose $f \in L^\infty_{h^{-1}}(\mathds{R}^n)$ is such that $m(D)f=0$ as a distribution, i.e.\
     \begin{gather}\label{eq:3.3}
        \scalp{f}{\widetilde m(D)\phi} = 0 \quad\text{for all\ } \phi \in C_c^\infty(\mathds{R}^n).
    \end{gather}
    If $\left\{\eta \in \mathds{R}^n \mid  m(\eta) = 0\right\} \subset \{0\}$, then $f$ coincides Lebesgue a.e.\ with a polynomial $p\in L^\infty_{h^{-1}}(\mathds{R}^n)$.

    Conversely, if $f$ coincides Lebesgue a.e.\ with a polynomial for every $f \in L^\infty(\mathds{R}^n)$ satisfying \eqref{eq:3.3}, then $\left\{\eta \in \mathds{R}^n \mid  m(\eta) = 0\right\} \subset \{0\}$.
\end{theorem}
Theorem \ref{th:3.1} can be used for functions of the form $h(x) = \Lambda^\beta(x)\log(e + |x|)^\alpha$ for $\alpha,\beta\geq 0$. If we use $h(x)=\log(e + |x|)$, we see that in the setting of Theorem \ref{th:3.1} every solution $m(D)f=0$ is a.e.\ constant.

We want to point out that without the boundedness condition \eqref{eq:3.1} the function $f$ is not necessarily a tempered distribution. For example, if we choose $h(x)=e^{a|x|^{\gamma}}$, then $f$ need not be a tempered distribution, and our method above would not work. In the next section we discuss functions, which might not define a tempered distributions, but are positive.
The condition that $\tilde m(D)$ maps $C_c^\infty(\rn)$ into $L_{h}^1(\mathds{R}^n)$ is essential for Theorem \ref{th:3.1}.
Recently we learned from M.\ Kwasnicki (private communication) and the paper by T.\ Grzywny \& M.\ Kwasnicki \cite[Theorem 1.9.(c) \& Theorem 4.1]{grz-kwa22}  that there is a multiplier given by a L\'{e}vy process, i.e.\ $m=\psi$, admitting a very slowly growing non-constant  function $u$ such that $\psi(D)u=0$. Note that this multiplier does not satisfy the mapping property required for Theorem \ref{th:3.1}.

\section{The Liouville theorem for rapidly growing functions}

We will now turn to the case where the solution of $\psi(D)f = 0$ is locally bounded and positive. It is well known that
\begin{gather*}
    f \ge 0 , \quad \Delta f = 0 \quad \implies \quad f \equiv \text{const}.
\end{gather*}
This is usually called the \textbf{strong Liouville property}. One cannot expect this property to hold for general Fourier multiplier operators discussed in Sections \ref{PolyBdd} and \ref{sec-slow} or even for higher order partial differential operators. Indeed,
\begin{gather*}
    f(x) := |x|^2 = x_1^2 + \cdots + x_n^2
\end{gather*}
is a non-constant nonnegative polynomially bounded solution of the equation $\Delta^2 f = 0$. Here, $\Delta^2 = \psi(D)$, $\psi(\xi) = |\xi|^4$, and $\left\{\xi \in \mathds{R}^n \mid \ \psi(\xi) = 0\right\} = \{0\} = \left\{\eta \in \mathds{R}^n \mid \ \psi(-i\eta) = 0\right\}$ (cf.\ Theorem \ref{th:4.3} below). So, we consider in this section Fourier multiplier operators $\psi(D)$ that generate positivity preserving operator semigroups $\left\{e^{-t \psi(D)}\right\}_{t\geq 0}$, i.e.\ such that $\psi$ are the characteristic exponents of L\'evy processes, see Section~\ref{sec-1} and Example~\ref{ex:1.5}. Even within this class of operators, the Laplacian is a special case, and the strong Liouville property does not hold for more general second order partial differential operators $\mathcal{L}$ without restrictions on the growth rate of a solution $f$ of the equation $\mathcal{L}f = 0$. Let
\begin{gather*}
    \mathcal{L} u(x) =  \frac12\, \nabla\cdot Q\nabla u(x) + b\cdot\nabla u(x) ,
\end{gather*}
where $Q \in \mathds{R}^{n\times n}$ is a positive semi-definite matrix, and $b \in \mathds{R}^n\setminus\{0\}$. Then $\mathcal{L}f = 0$ has non-constant nonnegative solutions. Indeed, if $Q \not= 0$, there exists $c_0 \in \mathds{R}^n$ such that $c_0\cdot Qc_0 > 0$. Since $b \not= 0$, there exists $c$ in a neighbourhood of $c_0$ such that $c\cdot Qc > 0$ and $b\cdot c \not= 0$. Let
\begin{gather*}
    f(x) := e^{\tau c\cdot x}, \quad\text{where}\quad \tau = - \frac{2b\cdot c}{c\cdot Qc} \not= 0 .
\end{gather*}
Then $f > 0$ is non-constant and
\begin{align*}
    \mathcal{L}f(x)
    &= \frac12\, \nabla\cdot Q\nabla e^{\tau c\cdot x} + b\cdot\nabla e^{\tau c\cdot x} = \left(\frac12\, \tau^2 c\cdot Qc + \tau b\cdot c\right) e^{\tau c\cdot x} \\
    &= \tau \left(\frac12\, \tau c\cdot Qc + b\cdot c\right) e^{\tau c\cdot x}  = 0 .
\end{align*}
If $Q = 0$, it is sufficient to take any $c \in \mathds{R}^n\setminus\{0\}$ such that $b\cdot c = 0$. Then
\begin{gather*}
    f(x) := e^{c\cdot x}
\end{gather*}
is positive, non-constant, and
\begin{gather*}
    \mathcal{L}f(x) = b\cdot\nabla e^{c\cdot x} =  (b\cdot c) e^{c\cdot x}  = 0 .
\end{gather*}

It follows from the above that one needs to put appropriate boundedness restrictions on $f$. Note also that while local boundedness of $f$ is sufficient
to ensure that $\scalp{f}{\widetilde \psi(D)\phi}$ is well-defined when $\psi(D)$ is a local operator, i.e.\ a partial differential operator, one needs boundedness restrictions on $f$ to ensure that $\psi(D)f$ has a meaning when $\psi(D)$ is non-local.

Our current proof is a refinement of our result in \cite[Theorem 17]{BS21}, and we focus here on the role of the upper bound. The key ingredient in the proof of \cite[Theorem 17]{BS21} is the equivalence of $-\psi(D)f=0$ and $e^{-t\psi(D)}f=f$, which is used to get a Choquet representation of all positive solutions $f\geq 0$. In order to define $\psi(D)f$ or $e^{-t\psi(D)}f$ as a distribution, we need a bound $f\leq g$ and an integrability condition on the measure $\nu$ appearing in the L\'evy--Khintchine representation \eqref{eq:1.2} of $\psi$, see the discussion in \cite{BS21}. Here we will concentrate on the role of the upper bound $g$ in the proof of the strong Liouville property which was glossed-over in our presentation in \cite[Theorem 17]{BS21}; this explains, in particular, in which directions $\psi$ can be extended from $\mathds{R}^n$ into $\mathds{C}^n = \mathds{R}^n + i\mathds{R}^n$.

Let $g : \mathds{R}^n\to [1, \infty)$ be a locally bounded, measurable submultiplicative function. We need to describe the directional growth behaviour of $g$. For $\omega \in \mathds{S}^{n - 1} := \left\{x \in \mathds{R}^n\mid |x| = 1\right\}$, let
\begin{gather*}
    g_\omega(r) := g(r\omega) , \quad r \ge 0 .
\end{gather*}
Set
\begin{gather*}
       \beta(\omega) := \inf_{r > 0} \frac{\ln g_\omega(r)}{r} = \lim_{r \to \infty} \frac{\ln g_\omega(r)}{r}\,.
\end{gather*}
Since $\ln g_\omega$ is subadditive, the infimum is, in fact, a limit; moreover, the function $\beta : \mathds{S}^{n - 1} \to \mathds{R}$ is continuous,
see \cite[Theorem 7.13.2]{HP57}.

Applying \cite[Chapter~II, Theorem~1.3]{KPS82} to the function $v(t) := g_\omega(\ln t)$, $t > 0$,
we conclude that
\begin{equation}\label{eq:4.1}
    g_\omega(r) \geq e^{\beta(\omega) r} \quad\text{for all\ } r > 0 ,
\end{equation}
and that for every $\epsilon > 0$, there is some $r_\epsilon>0$ such that
\begin{equation}\label{eq:4.2}
    g_\omega(r) \leq e^{(\beta(\omega) + \epsilon) r} \quad\text{for all\ } r > r_\epsilon.
\end{equation}
Let
\begin{equation}\label{Pi}
    \Pi_g := \left\{\xi \in \mathds{R}^n\mid \xi\cdot\omega \leq \beta(\omega) \quad\text{for all\ } \omega \in \mathds{S}^{n - 1}\right\} .
\end{equation}

\begin{example}\label{ex:4.2}
\begin{enumerate}
\item\label{ex:4.2-1}
    If $g(x) = (1 + |x|)^\lambda$ with $\lambda \geq 0$, or $g(x) = e^{\alpha |x|^\gamma}$ with $\alpha \geq 0$, $\gamma \in [0, 1)$, then $\beta(\omega) \equiv 0$ and $\Pi_g = \{0\}$.

\item\label{ex:4.2-2}
    If $g(x) = e^{\alpha |x|}$ with $\alpha > 0$, then $\beta(\omega) \equiv \alpha$ and
    \begin{gather*}
        \Pi_g = \left\{\eta \in \mathds{R}^n\mid |\eta| \leq \alpha\right\} .
    \end{gather*}

\item\label{ex:4.2-3}
    If $g(x) : = \max\left\{e^{x_1}, 1\right\}$, then $\beta(\omega) =  \max\left\{\omega_1, 0\right\}$, and it is easy to see that
    \begin{align*}
    \Pi_g
    &=  \left\{\xi \in \mathds{R}^n\mid \xi_1\omega_1 + \sum_{j = 2}^d \xi_j\omega_j \leq \max\left\{\omega_1, 0\right\} \quad\text{for all\ }  \omega \in \mathds{S}^{n - 1}\right\}
    \\
    &= \left\{\xi = (\xi_1, 0, \dots, 0) \in \mathds{R}^n\mid \xi_1\omega_1  \leq \max\left\{\omega_1, 0\right\} \quad\text{for all\ }  \omega_1 \in [-1, 1]\right\} \\
    &= \left\{\xi = (\xi_1, 0, \dots, 0) \in \mathds{R}^n\mid \xi_1 \in [0, 1]\right\}  ,
    \end{align*}
    i.e.\ $\Pi_g$ is the one-dimensional interval $[0,1]\times\{0\}\times\dots\times\{0\}$.
\end{enumerate}
\end{example}
We will need to extend $\psi$ into a strip in $\mathbb{C}^n$. This can be achieved by \eqref{eq:1.1} or \eqref{eq:1.2} provided that the measure $\nu$ is sufficiently well-behaved.
It follows from \eqref{eq:4.1} and \eqref{Pi} that if $\eta \in \Pi_g$, then
\[
\left|e^{i(\xi - i\eta)y}\right| = e^{\eta y} \le g(y) \quad\text{for all\ } y \in  \mathds{R}^n .
\]
So, if one assumes, as we do in the next theorem, that $\int_{|y| \geq 1} g(y) \,\nu(dy) < \infty$, then it follows from \eqref{eq:1.2}  that
$\psi(\xi -i\eta)$ is well-defined.
\begin{theorem}[strong Liouville property]\label{th:4.3}
    Let $\psi : \mathds{R}^n\to \mathds{C}$\ be the characteristic exponent of a L\'evy process with L\'evy triplet $(b,Q, \nu)$. Let $g : \mathds{R}^n\to [1, \infty)$ be a locally bounded, measurable submultiplicative function  such that $\int_{|y| \geq 1} g(y) \,\nu(dy) < \infty$. Every measurable, positive and $g$-bounded \textup{(}$0 \leq f \leq g$\textup{)} weak solution $f$ of the equation $\psi(D) f = 0$ is constant if, and only if,
    \begin{equation}\label{eq:4.4}
        \left\{\xi \in \mathds{R}^n\mid \psi(\xi) = 0\right\}
        = \{0\}
        = \left\{\eta \in \Pi_g\mid \psi(-i\eta) = 0\right\} .
    \end{equation}
\end{theorem}

\begin{proof}
\emph{Sufficiency of \eqref{eq:4.4}.} Assume that the first equality in \eqref{eq:4.4} holds, $0 \leq f \leq g$, and $\psi(D) f = 0$. Theorem \ref{th:4.3} is a refinement of the corresponding result \cite[Theorem 17]{BS21} that takes care of the nature of the extension of $\psi$ into the complex plane. Rather than reproducing the full proof here, we quote the structural part of the solution to $\psi(D)f=0$ from \cite[Theorem 17]{BS21}: there exists a measure $\rho$ with support in $E := \left\{\eta \in \mathds{R}^n\mid \psi(-i\eta) = 0\right\}$ such that
\begin{gather*}
    f(x) = \int_E e^{x\cdot\xi}\, \rho(d\xi) .
\end{gather*}
We will show that $\supp\rho \subset \left\{\eta \in \Pi_g\mid \psi(-i\eta) = 0\right\}$; thus, the second equality in \eqref{eq:4.4} proves that $f\equiv \rho(\{0\})$ a.s.

Suppose there exists some $\xi^0 \in \supp \rho\setminus \Pi_g$. This means that there is some $\omega^0 \in \mathds{S}^{n-1}$ such that $\xi^0\cdot\omega^0 > \beta(\omega^0)$. Take any $0 < \epsilon < \frac12(\xi^0\cdot\omega^0 - \beta(\omega^0))$ and consider the open ball $B_\epsilon := B_\epsilon(\xi^0)$ of radius $\epsilon$ centred at $\xi^0$. Let $x = r \omega^0$, $r > 0$. Then
\begin{gather*}
    x\cdot\xi
    = x\cdot\xi^0 + x\cdot(\xi - \xi^0)
    \geq (\xi^0\cdot\omega^0) r - r|\xi - \xi^0|
    > (\xi^0\cdot\omega^0 - \epsilon) r
    \quad\text{for all\ } \xi \in B_\epsilon,
\end{gather*}
and there exists some $r_0 > 0$, depending only on $g$, $\omega_0$, and $\epsilon$ (see \eqref{eq:4.2}), such that
\begin{gather*}
    \frac{e^{x\cdot\xi}}{g_{\omega^0}(r)}
    > \frac{e^{(\xi^0\cdot\omega^0 - \epsilon) r}}{g_{\omega^0}(r)}
    \geq \frac{e^{(\xi^0\cdot\omega^0 - \epsilon) r}}{e^{(\beta(\omega^0) + \epsilon) r}}
    = e^{(\xi^0\cdot\omega^0 - \beta(\omega^0) - 2\epsilon)r}
    \quad\text{for all\ } r \geq r_0 .
\end{gather*}
Since $\xi^0 \in \supp \rho$, we know that $\rho(B_\epsilon) > 0$. Thus,
\begin{align*}
    \frac{f(x)}{g(x)}
    = \frac{\int_E e^{x\cdot\xi}\, \rho(d\xi)}{g(x)}
    \geq \frac{\int_{B_\epsilon} e^{x\cdot\xi}\, \rho(d\xi)}{g_{\omega^0}(r)}
    &> \frac{e^{(\xi^0\cdot\omega^0 - \beta(\omega^0) - 2\epsilon)r}  g_{\omega^0}(r) \rho(B_\epsilon)}{g_{\omega^0}(r)}  \\
    &= e^{(\xi^0\cdot\omega^0 - \beta(\omega^0) - 2\epsilon)r}  \rho(B_\epsilon) \to \infty
    \quad\text{as\ } r \to \infty.
\end{align*}
This contradicts the bound $0 \leq f \leq g$, and we conclude that $\supp \rho\setminus \Pi_g = \emptyset$, hence $\supp\rho \subset \left\{\eta \in \Pi_g\mid \psi(-i\eta) = 0\right\}$.

\medskip\noindent
\emph{Necessity of \eqref{eq:4.4}.}
Suppose every measurable, positive and $g$-bounded  weak solution $f$ of $\psi(D) f = 0$ is constant. If the first equality in \eqref{eq:4.4} does not hold, then there exists $\xi^0 \in \left\{\xi \in \mathds{R}^d : \ \psi(\xi) = 0\right\}\setminus\{0\}$. Since $\psi(-\xi^0)=\overline{\psi(\xi^0)}=0$ and $\psi(0)=0$, one has $-\xi^0, 0 \in \left\{\xi \in \mathds{R}^d : \ \psi(\xi) = 0\right\}$. Let $e_{\pm\xi^0}(x) := e^{\pm i\xi^0\cdot x}$. Then $\psi(D)e_{\pm\xi^0} = \psi(\pm\xi^0) e_{\pm\xi^0} = 0$ and $\psi(D)1 = \psi(0) = 0$. Hence
\begin{gather*}
    f(x) := \frac12\big(1 + \cos(\xi^0\cdot x)\big)
\end{gather*}
is a non-constant measurable, positive and $g$-bounded  weak solution $f$ of $\psi(D) f = 0$. This contradiction proves the first equality in \eqref{eq:4.4}.

Suppose now the second equality in \eqref{eq:4.4} does not hold, i.e.\ there exists a nonzero $\theta \in \left\{\eta \in \Pi_g\mid \psi(-i\eta) = 0\right\}$. From the definition of $\Pi_g$ we see that, for every $x = r\omega$, $\omega \in \mathds{S}^{n-1}$ and $r \geq 0$,
\begin{gather*}
    f(x)
    := e^{\theta\cdot x}
    = e^{(\theta\cdot \omega) r}
    \leq e^{\beta(\omega) r}
    \leq g_\omega(r)
    = g(r\omega)
    = g(x),
\end{gather*}
see \eqref{eq:4.1}. So, $f$ is non-constant, positive and $g$-bounded. Since $\psi(D) f = \psi(-i\theta ) f \equiv 0$, we get again a contradiction.
\end{proof}

\section{Coupling}

In this section we want to establish a connection between our Livouville theorem (Theorem~\ref{th:1.1}), the coupling property of L\'evy processes, and the notion of space-time harmonic functions. Throughout, $(X_t)_{t\geq 0}$ is a L\'evy process starting at $0$ with characteristic exponent $\psi$, see Section~\ref{sec-1}.

A \textbf{coupling} of the L\'evy process $(X_t)_{t\geq 0}$ with values in $\real^n$ is any Markov process $((Z^x_t,Z^y_t))_{t\geq 0}$ taking values in $\real^{2n}$ such that each of the marginal processes $(Z_t^z)_{t\geq 0}$ has the same finite dimensional distributions as $(X_t^z)_{t\geq 0}$, $X_t^z:=X_t+z$, $z=x,y$. The coupling is usually realized on a new probability space $(\Omega, \mathbb{P}=\mathbb{P}^{(x,y)},\mathcal{F})$. The L\'evy process $(X_t)_{t\geq 0}$ has the \textbf{(exact) coupling property}, if for some coupling with $x\neq y$ the trajectories of the processes $(Z_t^x)_{t\geq 0}$ and $(Z_t^y)_{t\geq 0}$ meet with probability one in finite time; this is the \textbf{coupling time} $\tau=\tau^{x,y}$. Intuitively, both processes $(Z_t^x)_{t\ge 0}$ and $(Z_t^y)_{t\ge 0}$ run on the same probability space, move together and cannot (statistically) be distinguished from each other or from $(X_t+z)_{t\geq 0}$.

Coupling techniques provide powerful tools to study the regularity of the semigroup $x\mapsto \mathcal P_t f(x) = \Ee f(X_t^x)$, the existence of invariant (stationary) measures for $(X_t^x)_{t\geq 0, x\in\mathds{R}^n}$ and many further properties, see the discussion in  \cite{T00}.  It was shown in \cite[Theorem~4.1]{SW11} that a L\'evy process has the coupling property if, and only if, the transition probability $ p_t(dy)$ of the L\'evy process $(X_t)_{t\ge 0}$ has an absolutely continuous component for some $t\ge 0$.

If $(X_t^x)_{t\geq 0, x\in \mathds{R}^n}$ is a L\'evy process (or a general Markov process with generator $\mathcal{A}_x$), the space-time process $((s+t,X_t^x))_{t\geq 0, (s,x)\in [0,\infty)\times\mathds{R}^n}$ is again a L\'evy process (resp., Markov process), and its semigroup is given by $\mathcal Q_t u(s,x) = \Ee(s+t,X_t^x)$. Thus, the infinitesimal generator is of the form $\frac{d}{ds}-\psi(D_x)$ (resp.\ $\frac{d}{ds}+\mathcal{A}_x$), and we are naturally led to the notion of space-time harmonic functions and the space-time Liouville property.

\begin{definition}\label{def:5.1}
Let $(X_t^x)_{t\geq 0,x\in\mathds{R}^n}$ be a L\'evy process.
\begin{enumerate}
\item\label{def:5.1-1}
    A function $f:[0,\infty)\times \mathds{R}^n\to \mathds{R}$ is \textbf{space-time harmonic}, if one has $f(s,x)=\Ee^x f(s+t,X_t)=\Ee f(s+t,X_t+x)$ for \emph{every} $t,s\geq 0$ and \emph{every} $x\in\mathds{R}^n$.
\item\label{def:5.1-2}
    The process has the \textbf{space-time Liouville property}, if every measurable and bounded space-time harmonic function is constant.
\end{enumerate}
\end{definition}

\allowdisplaybreaks

\begin{remark}\label{re:5.2}
\begin{enumerate}
\item\label{re:5.2-1}
    The space-time Liouville property is known to be equivalent to the exact coupling property for Markov processes, see
    \cite[Theorem 4.1, p.~205]{T00}.
\item\label{re:5.2-2}
    We call a function $f:\mathds{R}^{n+1}\to \mathds{R}$ satisfying $f(s,x)=\Ee^x f(s+t,X_t)$ for every $t\geq 0$ and $(s,x)\in \mathds{R}^{n+1}$ also space-time harmonic.
\item\label{re:5.2-3}
    Notice that the notions of \enquote{space-time harmonicity} and \enquote{exact coupling} are pointwise defined notions, which do not allow for exceptional sets. This is the main difficulty when we want to compare our notion of the Liouville property and the space-time Liouville property.
\end{enumerate}
\end{remark}

For a L\'evy triplet $(b,Q,\nu)$ with L\'evy process $X=(X_t)_{t\ge 0}$ in $\mathds{R}^n$ and a vector $\eta\in \mathds{R}^n$ we introduce the notation
\begin{gather*}
    b^\eta := \eta\cdot b+\int \eta\cdot y\left(\I_{(0,1)}(|\eta\cdot y|)-\I_{(0,1)}(|y|)\right)\nu(dy).
\end{gather*}
This is the vector in the L\'evy triplet of $\eta\cdot X_t$, see \cite[Proposition 11.10]{sato}.
We need a few auxiliary lemmas preparing the proof of the main result of this section, Theorem~\ref{th:5.7}.
\begin{lemma}\label{le:5.3_0}
    Let $\psi:\mathds{R}^{d}\to\mathds{C}$  be the characteristic exponent of a L\'evy process with L\'evy triplet $(b,Q,\nu)$. If $\xi \in \mathds{R}^n$ is such that $\psi(\xi) \in i\mathds{R}$, then $\psi(\xi) = -i b^\xi$.
\end{lemma}
\begin{proof}
Let
\begin{align*}
    &\mathds{E}_\xi := \left\{y \in \mathds{R}^n \mid \ 1-e^{iy\cdot\xi} = 0\right\} = \left\{y \in \mathds{R}^n \mid \ y\cdot\xi \in 2\pi\mathds{Z}\right\}, \\
    &\mathds{E}_\xi^0 := \left\{y \in \mathds{R}^n \mid \ y\cdot\xi = 0\right\} = \{\xi\}^\perp.
\end{align*}
Since $\re\left(1-e^{iy\cdot\xi}\right) \ge 0$, and $\re\left(1-e^{iy\cdot\xi}\right) = 0$ if, and only if, $1-e^{iy\cdot\xi} = 0$, it follows from $\psi(\xi) \in i\mathds{R}$ and \eqref{eq:1.2} that $\nu(\mathds{R}^n\setminus \mathds{E}_\xi) = 0$ and $Q\xi=0$.

If $y \in \mathds{E}_\xi \setminus \mathds{E}_\xi^0$, then $|y\cdot\xi| \ge 2\pi > 1$. So, $\I_{(0,1)}(|\xi\cdot y|) = 0$ if $y \in \mathds{E}_\xi$.
Hence
\begin{align*}
    \psi(\xi)
    &= -ib\cdot\xi  + \int\limits_{\mathds{R}^n} \left[1-e^{iy\cdot\xi} + iy\cdot\xi\I_{(0,1)}(|y|)\right]\, \nu(dy)   \\
    &= -ib\cdot\xi  + \int\limits_{\mathds{E}_\xi} \left[1-e^{iy\cdot\xi} + iy\cdot\xi\I_{(0,1)}(|y|)\right]\, \nu(dy) \\
    &= -ib\cdot\xi  + \int\limits_{\mathds{E}_\xi} \left[1-e^{iy\cdot\xi} + iy\cdot\xi\left(\I_{(0,1)}(|y|) - \I_{(0,1)}(|\xi\cdot y|)\right) \right]\nu(dy) \\
    &= -ib\cdot\xi  + i \int\limits_{\mathds{E}_\xi} y\cdot\xi\left[\I_{(0,1)}(|y|) - \I_{(0,1)}(|\xi\cdot y|)\right] \nu(dy) \\
    &= -ib\cdot\xi  + i \int\limits_{\mathds{R}^n} y\cdot\xi\left[\I_{(0,1)}(|y|) - \I_{(0,1)}(|\xi\cdot y|)\right] \nu(dy) = -i b^\xi .
\qedhere
\end{align*}
\end{proof}

\begin{lemma}\label{le:5.3}
    Let $n,d\in\mathds{N}$,\, $\psi_1:\mathds{R}^{d}\to\mathds{C}$ and $\psi_2:\mathds{R}^n\to\mathds{C}$ be characteristic exponents of two L\'evy processes with L\'evy triplets $(b_1,Q_1,\nu_1)$ and $(b_2,Q_2,\nu_2)$. Then
\begin{gather*}
    \left\{(\eta,\xi)\in \mathds{R}^{d}\times\mathds{R}^n \mid \psi_1(\eta)+\psi_2(\xi)=0\right\}
    = \left\{\eta\in \psi_{1}^{-1}(i\mathds{R}),\; \xi\in\psi_2^{-1}(i\mathds{R}) \;\middle|\; b_{1}^\eta+b_{2}^\xi=0\right\}.
\end{gather*}
\end{lemma}

\begin{proof}
Since $\re \psi_i\geq 0$ for $i=1,2$ (see \eqref{eq:1.2}), a necessary condition for $\psi_1(\eta)+\psi_2(\xi)=0$ is $\re \psi_1(\eta)=0$ and $\re\psi_2(\xi)=0$. In this case, $\psi_1(\eta)=-i b_1^\xi $ and $\psi_2(\xi)=-i b_2^\xi$ (see Lemma \ref{le:5.3_0}), and the equality $\psi_1(\eta)+\psi_2(\xi)=0$ is equivalent to $b_{1}^\eta+b_{2}^\xi=0$.
\end{proof}

\begin{remark}\label{re:5.4}
  Assume that $d = 1$ and $\psi_1(\eta)=-ib \eta$ for some $b\in \mathds{R}\setminus\{0\}$. We see easily that
  the set $\{(\eta,\xi)\in \mathds{R}\times\mathds{R}^n \mid \psi_1(\eta)+\psi_2(\xi)=0\}$ is equal to $\{(0,0)\}$ if, and only if, $\psi_2^{-1}(i\mathds{R})=\{0\}$.
\end{remark}

\begin{lemma}\label{le:5.5}
    Let $f:[0,\infty)\times \mathds{R}^n\to\mathds{R}$ be a space-time harmonic function for the L\'evy process $(X_t)_{t\geq 0}$. Then there exists a unique extension $\tilde f:(-\infty,\infty)\times\mathds{R}^n\to\mathds{R}$, which is still space-time harmonic.
\end{lemma}
\begin{proof}
Assume that $f:[0,\infty)\times \mathds{R}^n\to \mathds{R}$ is such that
\begin{align*}
    f(s,x)
    = \Ee f(s+t,x+X_t)\quad\textrm{for all }s,t\geq 0,\; x\in\mathds{R}^n.
\end{align*}
We define for $s>0$ and $x\in \mathds{R}^n$
\begin{align*}
    f(-s,x) := \Ee f(0,x+X_{s}).
\end{align*}
Let $s > t>0$. By the Markov property and the stationary increments property of $X_t$ we obtain that
\begin{align*}
    \Ee f(-s+t,x+X_t)
    = \int \Ee\left(f(0,(x+ X_t)+ y)\right)\big|_{y=X_{s-t}} \,d\Pp
    &= \Ee f(0,(x+X_t)+(X_s-X_t))\\
    &= \Ee f(0,x+X_s) =: f(-s,x).
\end{align*}
Now let $t>s>0$. Again by the Markov property and the stationarity of the increments we see that
\begin{gather*}
    \Ee f(-s+t,x+X_t)
    = \Ee f(t-s,x+ (X_t-X_s)+X_s)
     = \Ee f(0,x+X_s)
    =: f(-s,x).
\qedhere
\end{gather*}
\end{proof}

Recall that a L\'evy process has the strong Feller property -- i.e.\ $x\mapsto \Ee u(x+X_t)$ is a continuous function for every bounded measurable $u$ -- if, and only if, the transition probability $\Pp(X_t\in dy)$ is absolutely continuous w.r.t.\ Lebesgue measure, cf.\ Jacob \cite[Lemma~4.8.20]{jac-1}.
\begin{lemma}\label{le:5.6}
    Let $(X_t)_{t\geq 0}$ be a L\'evy process with characteristic exponent $\psi_X $ and let $(S_t)_{t\geq 0}$ be an independent subordinator\footnote{A \textbf{subordinator} is a one-dimensional L\'evy process with increasing sample paths.}%
    with Laplace transform $\Ee e^{-xS_t}=e^{-tf_S(x)}$, $x\geq 0$. If both $X$ and $S$ are strong Feller processes, then the subordinated space-time process $(S_t,X_{S_t})$ is a strong Feller L\'{e}vy process; its characteristic exponent is given by
    \begin{align}\label{eq:5.1}
        f_S(-i\tau + \psi_X(\xi)),\quad (\tau,\xi)\in\mathds{R}\times\mathds{R}^n.
    \end{align}
\end{lemma}
\begin{proof}
    The process $Y=(Y_t)_{t\ge0}:=((S_t,X_{S_t}))_{t\geq 0}$ can be seen as subordination of the space-time L\'{e}vy process $((t,X_t))_{t\geq 0}$, from which we conclude that $Y$ is indeed a L\'evy process. Furthermore, $Y_t$ has for every $t>0$ a transition density $p_{Y_t}(s,x)$, which is given by
    \begin{align*}
        p_{Y_t}(s,x) = p_{X_s}(x)p_{S_t}(s)\I_{[0,\infty)}(s).
\end{align*}
    This implies that $Y$ has the strong Feller property. That the symbol of $Y$ is given by \eqref{eq:5.1} is a direct consequence of the subordination of the process $((t,X_t))_{t\geq 0}$:
    \begin{gather*}
        \Ee \left[e^{i\tau S_t + i\xi\cdot X_{S_t}}\right]
        = \Ee \left[\left.\left(\Ee e^{i\tau r + i\xi X_r}\right)\right|_{r=S_t}\right]
        = \Ee \left[e^{-(-i\tau+\psi_X(\xi)) S_t}\right]
        = e^{-t f_S(-i\tau+\psi_X(\xi))}.
    \qedhere
    \end{gather*}
\end{proof}
\begin{theorem}\label{th:5.7}
    Let $(X_t)_{t\geq 0}$ be a strong Feller L\'evy process with characteristic exponent $\psi_X$. Then the following assertions are equivalent:
    \begin{enumerate}
    \item\label{th:5.7-1} $(X_t)_{t\geq 0}$ has the \textup{(}exact\textup{)} coupling property,
    \item\label{th:5.7-2} $(X_t)_{t\geq 0}$ has the space-time Liouville property,
    \item\label{th:5.7-3} $(t,X_t)_{t\geq 0}$ has the Liouville property \textup{(}as in Theorem~\ref{th:1.1}\textup{)},
    \item\label{th:5.7-4} $\{\xi\in\mathds{R}^n \mid \psi_X(\xi)\in i\mathds{R}\}=\{0\}$.
\end{enumerate}
\end{theorem}
\begin{proof}
\ref{th:5.7-1}$\Leftrightarrow$\ref{th:5.7-2} is due to Thorisson~\cite[Theorem~4.5, p.~205]{T00}.

\ref{th:5.7-3}$\Leftrightarrow$\ref{th:5.7-4} is due to Theorem~\ref{th:1.1}, (the proof of) Lemma~\ref{le:5.6} for $f_s(\lambda)=\lambda$ and Remark~\ref{re:5.4}.

\ref{th:5.7-2}$\Rightarrow$\ref{th:5.7-3}: let $u$ be a bounded measurable function such that $\left(\frac{d}{ds}-\psi_X(D_x)\right)u=0$ in the sense of distributions. If $\mathcal{P}_t$ is the semigroup generated by $\frac{d}{ds}-\psi(D_x)$, we know from the relation between semigroup and generator that
\begin{gather*}
    \mathcal{P}_t u(s,x) = u(s,x) + \int_0^t \mathcal{P}_r\left(\frac{d}{ds}-\psi_X(D_x)\right)u(s,x)\,dr = u(s,x)
\end{gather*}
$t>0$ and all $(s,x)\in\mathds{R}\times\mathds{R}^n$ in the sense of distributions.
Since $u(s,x) = \Ee u(s+t,x+X_t)$ does not depend on $t>0$, we have
\begin{align*}
    u(s,x)
    = \int_{0}^\infty \underbrace{\Ee u(s+r,x+X_\tau)}_{=u(s,x)}\,\Pp(S_t\in dr)
    = \Ee u(s+S_{t},x+X_{S_t}),
\end{align*}
where $S=(S_t)_{t\geq 0}$ is a $1/2$-stable subordinator.\footnote{A $1/2$-stable subordinator is a subordinator with Laplace exponent $f_S(x)=\sqrt{x}$. The transition probability of the random variable $S_t$, $t>0$, is given by the L\'evy distribution $r\mapsto t(2\pi )^{-1/2} r^{-3/2} e^{-t^2/2r}$, $r>0$, see \cite[Example 40.14]{sato}.} By Lemma \ref{le:5.6}, $(S_t,X_{S_t})$ has again the strong Feller property and hence, we choose $u$ to be continuous, and the equality  $\mathcal{P}_tu(s,x)=u(s,x)$ holds pointwise for every $(s,x)\in\mathds{R}\times\mathds{R}^n$; in particular, $u$ is a bounded and continuous function. Setting $f(t,x):= u(t,x)$ it is clear that $f$ is space-time harmonic, hence constant.

\ref{th:5.7-3}$\Rightarrow$\ref{th:5.7-2}:
We have to show that any space-time harmonic function $f:[0,\infty)\times \mathds{R}^n\to\mathds{R}$ is constant. Fix such an $f$ and construct, as in Lemma~\ref{le:5.5}, its unique extension to the negative real line; this extension is still space-time harmonic.
By a similar argument as before we see that $f(s,x)=\Ee f(s+S_{t},x+X_{S_t})$ pointwise for every $(s,x)\in \mathds{R}\times\mathds{R}^n$, where $S=(S_t)_{t\ge 0}$ is again a $1/2$-stable subordinator. We conclude that $f$ is continuous. The characteristic exponent of the process $(S_{t},X_{S_t})$ is given by $\sqrt{-i\tau + \psi_X(\xi)}$ by Lemma \ref{le:5.6}. As $(t,X_t)_{t\ge 0}$ has the Liouville property, we know that $i\tau - \psi_X(\xi)=0$ if, and only if, $(\tau,\xi)= (0,0)$, from which we conclude that $\sqrt{-i\tau + \psi_X(\xi)}=0$ if, and only if, $(\tau,\xi)= (0,0)$. In view of Theorem~\ref{th:1.1}, $(S_t,X_{S_t})_{t\geq 0}$ has the Liouville property; therefore, $f$ is a.e.\ constant, and, as $f$ is continuous, $f$ is constant.
\end{proof}

\end{document}